\documentclass[a4paper,11pt]{article}

\usepackage{braket,hyperref}
\usepackage{amssymb,mathtools,amsthm,amsmath}
\usepackage{xargs}
\usepackage{enumitem}
\usepackage[colorinlistoftodos,prependcaption,textsize=tiny]{todonotes}
\usepackage{verbatim}

\theoremstyle{plain}
\newtheorem{theorem}{\bf Theorem}
\newtheorem{claim}[theorem]{\bf Claim}
\newtheorem{conjecture}[theorem]{\bf Conjecture}
\newtheorem{proposition}[theorem]{\bf Proposition}
\newtheorem{corollary}[theorem]{\bf Corollary}
\newtheorem{lemma}[theorem]{\bf Lemma}
\theoremstyle{definition}

\newenvironment{remark}[1][Remark.]{\begin{trivlist}
		\item[\hskip \labelsep {\bfseries #1}]}{\end{trivlist}}

\numberwithin{theorem}{section} 
\numberwithin{equation}{section}

\newcommandx{\error}[2][1=]{\todo[linecolor=red,backgroundcolor=red!25,bordercolor=red,#1]{#2}}
\newcommandx{\improvement}[2][1=]{\todo[linecolor=blue,backgroundcolor=blue!25,bordercolor=blue,#1]{#2}}
\newcommandx{\simania}[2][1=]{\todo[linecolor=green,backgroundcolor=green!25,bordercolor=green,#1]{#2}}

\newcommand{\Rea}{{\mathbb R}}

\newcommand{\simplex}[1]{\Delta_{#1}}
\newcommand{\skeleton}[2]{\Delta_{#1}^{(#2)}}

\DeclareMathOperator{\Spec}{Spec}
\newcommand{\spec}[2]{\Spec_{#1} (#2)}
\newcommand{\laplacian}[2]{L_{#1}\left(#2\right)}

\newcommand{\lap}[1]{L_{#1}}
\newcommand{\mineig}[2]{\mu_{#1}(#2)}

\newcommand{\cohomology}[2]{\tilde{H}^{#1}\left(#2;\Rea\right)}

\DeclareMathOperator{\lk}{lk}

\newcommand{\degree}[1]{\deg_{#1}}

\newcommand{\faces}[2]{#2(#1)}
\newcommand{\cochains}[2]{C^{#1}(#2)}
\newcommand{\sgn}[2]{(#1 : #2)}
\newcommand{\basiselement}[1]{1_{#1}}
\newcommand{\standardbasis}[1]{\{\basiselement{\sigma}\}_{\sigma\in #1}}
\newcommand{\orderedunion}[2]{[#1,#2]}

\newcommand{\cobound}[2]{d_{#1}#2}
\newcommand{\bound}[2]{d_{#1}^*#2}

\DeclareMathOperator{\Ker}{Ker}
\DeclareMathOperator{\Ima}{Im}
\newcommand{\inner}[2]{\left\langle#1,#2\right\rangle}
\newcommand{\norm}[1]{\left\|#1\right\|}

\newcommand{\cupdot}{\mathbin{\mathaccent\cdot\cup}}

\newcommand{\tempset}{Q}

\begin{document}
	\title{Spectral gaps, missing faces and minimal degrees}
\author{Alan Lew\footnote{Department of Mathematics, Technion, Haifa 32000, Israel.		
		 e-mail: alan@campus.technion.ac.il .  Supported by ISF grant no. 326/16.}}
	
	\date{}
	\maketitle

\begin{abstract}
	Let $X$ be a simplicial complex with $n$ vertices. A missing face of $X$ is a simplex $\sigma\notin X$ such that $\tau\in X$ for any $\tau\subsetneq \sigma$.
	 For a $k$-dimensional simplex $\sigma$ in $X$, its degree in $X$ is the number of $(k+1)$-dimensional simplices in $X$ containing it. Let $\delta_k$ denote the minimal degree of a $k$-dimensional simplex in $X$.  Let $\lap{k}$ denote the $k$-Laplacian acting on real $k$-cochains of $X$ and let $\mineig{k}{X}$ denote its minimal eigenvalue. We prove the following lower bound on the spectral gaps $\mu_k(X)$, for complexes $X$ without missing faces of dimension larger than $d$:
		\[
		\mineig{k}{X}\geq (d+1)(\delta_k+k+1)-d n.
		\]
	As a consequence we obtain a new proof of a vanishing result for the homology of simplicial complexes without large missing faces.	
	We present a family of examples achieving equality at all dimensions, showing that the bound is tight. For $d=1$ we characterize the equality case.	
\end{abstract}	
	
{\bf Keywords:} High dimensional Laplacian, Simplicial cohomology.
	
\section{Introduction}

Let $X$ be a finite simplicial complex on vertex set $V$. For $k\geq -1$ let $X(k)$ denote the set of all $k$-dimensional simplices of $X$, let $\cochains{k}{X}$ be the space of real valued $k$-cochains of $X$ and let $\cobound{k}:\cochains{k}{X}\to\cochains{k+1}{X}$ be the coboundary operator.
The reduced $k$-dimensional Laplacian of $X$ is defined by 
\[
\lap{k}(X)=\cobound{k-1}\bound{k-1}+\bound{k}\cobound{k}.
\]
$\lap{k}$ is a positive semi-definite operator from $\cochains{k}{X}$ to itself. The $k$-th \emph{spectral gap} of $X$, denoted by $\mineig{k}{X}$, is the smallest eigenvalue of $\lap{k}$.

A \emph{missing face} of $X$ is a subset $\sigma\subset V$ such that $\sigma\notin X$ but $\tau\in X$ for any $\tau\subsetneq \sigma$.
Let $h(X)$ denote the maximal dimension of a missing face of $X$.
For example, $h(X)=1$ if and only if $X$ is the clique complex of a graph $G$ (the missing faces of $X$ are the edges of the complement of $G$).

Let $\sigma\in X(k)$. The \emph{degree} of $\sigma$ in $X$ is defined by
\[
	\degree{X}(\sigma)=\left| \{ \eta\in X(k+1) : \, \sigma\subset \eta\} \right|.
\]
Let $\delta_k=\delta_k(X)$ be the minimal degree of a $k$-dimensional simplex, that is,
\[
	\delta_k= \min_{\sigma\in X(k)} \degree{X}(\sigma).
\]

Our main result is the following lower bound on the spectral gaps of $X$:

\begin{theorem}
	\label{thm:gersgorin_for_laplacian}
	Let $X$ be a simplicial complex on vertex set $V$ of size $n$, with $h(X)=d$. Then for $k\geq -1$,
		\[
		\mineig{k}{X}\geq (d+1)(\delta_k+k+1)-d n.
		\]
\end{theorem}

As a consequence we obtain a new proof of the following known result (see \cite[Prop. 5.4]{adamaszek2014extremal}):
\begin{theorem}
	\label{cor:gersgorin_for_laplacian}
		Let $X$ be a simplicial complex on vertex set $V$ of size $n$, with $h(X)=d$. Then
	$\cohomology{k}{X}=0$ for all $k>\frac{d}{d+1}{n}-1$.
\end{theorem}

The proof of Theorem \ref{thm:gersgorin_for_laplacian} relies on two main ingredients. The first one is the following theorem of Ger\v sgorin (see \cite[Chapter 6]{horn2012matrix}).

\begin{theorem}[Ger\v sgorin circle theorem]
	\label{thm:gers}
	Let $A\in\mathbb{C}^{n\times n}$ and $\lambda\in \mathbb{C}$ be an eigenvalue of $A$. Then there is some $i\in[n]$ such that
	\[
	\left| \lambda-A_{ii}\right|\leq \sum_{j\neq i} \left|A_{ij}\right|.
	\]	
\end{theorem}

The second ingredient is the following inequality concerning sums of degrees of simplices in $X$, which generalizes a known result for clique complexes (see \cite[Claim 3.4]{aharoni2005eigenvalues}, \cite{aigner2001turan}):
\begin{lemma}
	\label{lem:countdegrees_ver2} 
	Let $X$ be a simplicial complex on vertex set $V$ of size $n$, with $h(X)=d$.
Let $k\geq 0$ and $\sigma\in X(k)$. Then
	\[
	\sum_{\tau\in\faces{k-1}{\sigma}} \degree{X}(\tau)-(k-d+1)\degree{X}(\sigma)
	\leq d n-(d-1)(k+1).
	\]
\end{lemma}
Lemma \ref{lem:countdegrees_ver2} is closely related to Lemma $2.8$ in \cite{lew2018spectral}. For completeness, we include a full proof in Section \ref{section:sumofdegrees}.

Let $X$ and $Y$ be two simplicial complexes on disjoint vertex sets. The \emph{join} of $X$ and $Y$ is the complex
\[
X\ast Y= \{ \sigma\cupdot\tau : \, \sigma\in X,\, \tau\in Y \}.
\]
We will denote by $X\ast X$ the join of $X$ with a disjoint copy of itself. Also, we will denote the complex $X\ast X \ast \cdots\ast X$ ($k$ times) by $X^{\ast k}$.

Let $\Delta_{m}$ be the complete simplicial complex on $m+1$ vertices, and let $\Delta_{m}^{(k)}$ be its $k$-dimensional skeleton, i.e. the complex whose simplices are all the sets $\sigma\subset[m+1]$ such that $|\sigma|\leq k+1$.
The following example shows that the inequalities in Theorem \ref{thm:gersgorin_for_laplacian} are tight:

	Let $Z=\left(\skeleton{d}{d-1}\right)^{\ast t}\ast \simplex{r-1}$. Note that all the missing faces of $Z$ are of dimension $d$, and $\dim(Z)=dt+r-1$.
	Let $n=(d+1)t+r$ be the number of vertices of $Z$. We have:

\begin{proposition}
	\label{prop:extremalexample}
	\[
	\mineig{k}{Z}=\begin{cases}
	(d+1)\left(t-\left\lfloor \frac{k+1}{d} \right\rfloor\right)+r & \text{ if } -1\leq k\leq dt-1, \\
	r & \text{ if }\,\,\,\,\, dt\leq k\leq dt+r-1,
	\end{cases}
	\]
and
\[
\delta_k(Z)=\begin{cases}
n-(k+1)-\left\lfloor \frac{k+1}{d} \right\rfloor & \text{ if } -1\leq k\leq dt-1, \\
\,\,\,\,n-(k+1)-t & \text{ if }\,\,\,\,\, dt\leq k\leq dt+r-1.
\end{cases}
\]
In particular,
\[
	\mineig{k}{Z}=(d+1)(\delta_k(Z)+k+1)-dn
\]
for all $-1\leq k\leq \dim(Z)$.
\end{proposition}

Now we look at the case $d=1$. If $X$ is a clique complex with $n$ vertices, then by Theorem \ref{thm:gersgorin_for_laplacian} we have $\mineig{k}{X}\geq 2(\delta_k+k+1)-n$ for all $k\geq -1$, and we found a family of examples achieving equality in all dimensions. In particular, at the top dimension $k_t=\dim(Z)$ we obtain $\mineig{k_t}{Z}=2(k_t+1)-n$. The next proposition shows that these are the only examples achieving such an equality:
	
	\begin{proposition}
		\label{thm:gers_for_laplacian_eq_case}
		Let $X$ be a clique complex on vertex set $V$ of size $n$, such that   $\mineig{k}{X}=2(k+1)-n$ for some $k$. Then
		\[
		X\cong\left(\skeleton{1}{0}\right)^{*(n-k-1)}*\simplex{2(k+1)-n-1},
		\]
		(and in particular, $\dim(X)=k$).
	\end{proposition}

The paper is organized as follows: In Section \ref{section:prelim} we recall some definitions and results on simplicial cohomology and high dimensional Laplacians that we will use later. In Section \ref{section:sumofdegrees} we prove Lemma \ref{lem:countdegrees_ver2}.
Section \ref{section:main} contains the proofs of Theorems \ref{thm:gersgorin_for_laplacian} and \ref{cor:gersgorin_for_laplacian}.
 In Section \ref{section:extremal} we prove Propositions \ref{prop:extremalexample} and \ref{thm:gers_for_laplacian_eq_case}.

\section{Preliminaries}
\label{section:prelim}

Let $X$ be a simplicial complex on vertex set $V$, where $|V|=n$. 
An \emph{ordered simplex} is a simplex with a linear order of its vertices.
For two ordered simplices $\sigma$ and $\tau$ denote by $\orderedunion{\sigma}{\tau}$ their ordered union. For $v\in V$ denote by $v\sigma$ the ordered union of $\{v\}$ and $\sigma$.

For
$\tau\subset \sigma$, both given an order on their vertices, we define $\sgn{\sigma}{\tau}$ to be the sign of the permutation on the vertices of $\sigma$ that maps the ordered simplex $\sigma$ to the ordered simplex $\orderedunion{\sigma\setminus\tau}{\tau}$ (where the order on the vertices of $\sigma\setminus\tau$ is the one induced by the order on $\sigma$).

A \emph{simplicial $k$-cochain} is a real valued skew-symmetric function on all ordered $k$-dimensional simplices. That is, $\phi$ is a $k$-cochain if for any two ordered $k$-dimensional simplices $\sigma,\tilde{\sigma}$ in $X$ that are equal as sets, it satisfies $\phi(\tilde{\sigma})=\sgn{\tilde{\sigma}}{\sigma} \phi(\sigma)$.

For $k\geq -1$ let $X(k)$ be the set of $k$-dimensional simplices of $X$, each given some fixed order on its vertices. Let $\cochains{k}{X}$ denote the space of $k$-cochains on $X$. For $k=-1$ we have $X(-1)=\{\emptyset\}$, so we can identify $\cochains{-1}{X}=\Rea$.

For $\sigma\in \faces{k}{X}$, let
\[
\lk(X,\sigma)=\{ \tau\in X: \tau\cup \sigma\in X, \tau\cap\sigma=\emptyset\}
\]
be the \emph{link} of $\sigma$ in $X$.
For $U\subset V$, let $X[U]=\{\sigma\in X: \sigma\subset U\}$ be the subcomplex of $X$ induced by $U$.

The \emph{coboundary operator} $\cobound{k}{}: \cochains{k}{X}\to\cochains{k+1}{X}$ is the linear operator defined by
\[
\cobound{k}{\phi}(\sigma)=\sum_{\tau\in\sigma(k)}\sgn{\sigma}{\tau} \phi(\tau),
\]
where $\sigma(k)\subset \faces{k}{X}$ is the set of all $k$-dimensional faces of $\sigma$.

For $k=-1$ we have, under the identification $\cochains{-1}{X}=\Rea$, $\cobound{-1}{a}(v)=a$ for every $a\in \Rea$, $v\in V$.

Let $\cohomology{k}{X}=\Ker(\cobound{k}{})/\Ima(\cobound{k-1}{})$ be the \emph{$k$-th reduced cohomology group} of $X$ with real coefficients.

We define an inner product on $\cochains{k}{X}$ by
\[
\inner{\phi}{\psi}=\sum_{\sigma\in \faces{k}{X}} \phi(\sigma)\psi(\sigma).
\]
This induces a norm on $\cochains{k}{X}$:
\[
\norm{\phi}=\left(\sum_{\sigma\in\faces{k}{X}} \phi(\sigma)^2\right)^{1/2}.
\]
Let $\bound{k}{}: \cochains{k+1}{X}\to\cochains{k}{X}$ be the adjoint of $\cobound{k}{}$ with respect to this inner product.

Let $k\geq 0$. The \emph{reduced $k$-Laplacian} of $X$ is the positive semi-definite operator on $\cochains{k}{X}$ given by $\lap{k}=\cobound{k-1}{}\bound{k-1}{}+\bound{k}{}\cobound{k}{}$.

For $k=-1$ define $\lap{-1}=\bound{-1}{}\cobound{-1}{}: \Rea\to \Rea$. We have $\lap{-1}(a)= n$ for all $a\in\Rea$.


Let $\sigma\in \faces{k}{X}$. We define the $k$-cochain $\basiselement{\sigma}$ by
\[
\basiselement{\sigma}(\tau)=\begin{cases}
\sgn{\sigma}{\tau} & \text{ if $\sigma=\tau$ (as sets)},\\
0 & \text{ otherwise.}
\end{cases}
\]
The set $\standardbasis{\faces{k}{X}}$ forms a basis of the space $\cochains{k}{X}$.
We identify $L_k$ with its matrix representation with respect to the basis $\standardbasis{\faces{k}{X}}$. We denote the matrix element of $L_k$ at index $(\basiselement{\sigma},\basiselement{\tau})$ by $L_k(\sigma,\tau)$.

We can write the matrix $L_k$ explicitly (see e.g. \cite{duval2002shifted,goldberg2002combinatorial}):
\begin{claim}
	\label{claim:lapmatrix}	
	For $k\geq 0$
	\[
	\lap{k}(\sigma,\tau)=
	\begin{cases}
	\degree{X}(\sigma)+k+1 & \mbox{if } \sigma = \tau, \\
	\sgn{\sigma}{\sigma \cap \tau}\cdot\sgn{\tau}{\sigma \cap \tau}
	& \mbox{if } \left|\sigma \cap \tau \right|=k,  \sigma \cup \tau \notin X(k+1), \\
	0 & \mbox{otherwise.}
	\end{cases}	
	\]	
\end{claim}

An important property of the Laplacian operators is their relation to the cohomology of the complex $X$, first observed by Eckmann in \cite{eckmann1944harmonische}:

\begin{theorem}[Simplicial Hodge theorem]
	\[
	\cohomology{k}{X}\cong \Ker(\lap{k}).
	\]
\end{theorem}

In particular we obtain
\begin{corollary}
	\label{cor:hodge}
	$\cohomology{k}{X}=0$ if and only if $\mineig{k}{X}>0$.
\end{corollary}

Let $\spec{k}{X}$ be the spectrum of $\laplacian{k}{X}$, i.e. a multiset whose elements are the eigenvalues of the Laplacian.
The following theorem allows us to compute the spectrum of the join of simplicial complexes (see \cite[Theorem 4.10]{duval2002shifted}).

\begin{theorem}
	\label{thm:joinlaplacian}
	Let $X=X_1\ast\cdots\ast X_m$. Then
	\[
	\spec{k}{X}=\bigcup_{\substack{i_1+\ldots+i_m=k-m+1,\\ -1\leq i_j \leq \dim(X_j) \,\,\forall j\in[m]}} \spec{i_1}{X_1} +\cdots + \spec{i_m}{X_m},
	\]
\end{theorem}

We will need the following well known result on the spectrum of the complex $\skeleton{n-1}{k}$ (see e.g. \cite[Lemma 8]{gundert2016eigenvalues}):
\begin{claim}
	\label{claim:complete_skeleton_spectrum}
	\[
	\spec{i}{\skeleton{n-1}{k}}=\begin{cases}
	\{\underbrace{n,n,\ldots,n}_{\binom{n}{i+1} \mbox{ times}}\} & \mbox{if } -1\leq i \leq k-1, \\
	\{\underbrace{0,0,\ldots,0}_{\binom{n-1}{k+1} \mbox{ times}},\underbrace{n,n,\ldots,n}_{\binom{n-1}{k} \mbox{ times}}\} & \mbox{if } i=k.
	\end{cases}
	\]
\end{claim}

\section{Sums of degrees}
\label{section:sumofdegrees}
In this section we prove Lemma \ref{lem:countdegrees_ver2}.

\begin{claim}
	\label{lem:countdegrees_part1}
		Let $X$ be a simplicial complex on vertex set $V$.
	Let $k\geq 0$ and  $\sigma\in X(k)$. Then
\begin{multline*}
\sum_{\tau\in\faces{k-1}{\sigma}}\degree{X}(\tau)=(k+1)(\degree{X}(\sigma)+1)\\+\sum_{\substack{v\in V \setminus\sigma,\\ v\notin\lk(X,\sigma)}} \left|\{\tau\in \sigma(k-1) : \, v\in\lk(X,\tau)\}\right|.
\end{multline*}
\end{claim}
\begin{proof}
		\begin{multline}
		\label{eq:degsum}
		\sum_{\tau\in\faces{k-1}{\sigma}} \degree{X}(\tau)= \sum_{\tau\in \sigma(k-1)} \sum_{v\in \lk(X,\tau)} 1
		=\sum_{v\in V} \sum_{\substack{\tau\in \sigma(k-1),\\ \tau\in \lk(X,v)}}1\\
		=\sum_{v\in \sigma}\sum_{\substack{\tau\in \sigma(k-1),\\ \tau\in \lk(X,v)}}1+
		\sum_{v\in\lk(X,\sigma)} \sum_{\substack{\tau\in \sigma(k-1),\\ \tau\in \lk(X,v)}}1+
		\sum_{\substack{v\in V\setminus\sigma,\\ v\notin\lk(X,\sigma)}}\sum_{\substack{\tau\in \sigma(k-1),\\ \tau\in \lk(X,v)}}1.
		\end{multline}
		We consider separately the first two summands on the right hand side of \eqref{eq:degsum}:
		\begin{enumerate}
			\item For $v\in\sigma$, there is only one $\tau\in\faces{k-1}{\sigma}$ such that $\tau\in\lk(X,v)$, namely $\tau=\sigma\setminus\{v\}$. Thus the first summand is $k+1$.
			
			\item For $v\in\lk(X,\sigma)$, any $\tau\in\faces{k-1}{\sigma}$ is in $\lk(X,v)$, therefore the second summand is $(k+1)\degree{X}(\sigma)$. 
			
		\end{enumerate}
	Thus we obtain
	\begin{multline*}
	\sum_{\tau\in\faces{k-1}{\sigma}}\degree{X}(\tau)=(k+1)(\degree{X}(\sigma)+1)\\+\sum_{\substack{v\in V \setminus\sigma,\\ v\notin\lk(X,\sigma)}} \left|\{\tau\in \sigma(k-1) : \, v\in\lk(X,\tau)\}\right|.
	\end{multline*}
\end{proof}

\begin{proof}[Proof of Lemma \ref{lem:countdegrees_ver2}]
		By Claim \ref{lem:countdegrees_part1} we have
		\[
		\sum_{\tau\in\faces{k-1}{\sigma}} \degree{X}(\tau)\\
		=(k+1)(\degree{X}(\sigma)+1)+\sum_{\substack{v\in V\setminus\sigma,\\ v\notin\lk(X,\sigma)}}\sum_{\substack{\tau\in \sigma(k-1),\\ \tau\in \lk(X,v)}}1.
		\]
		Let $v\in V\setminus \sigma$ such that $v\notin \lk(X,\sigma)$. For each $\tau\in \faces{k-1}{\sigma}$ such that $\tau\in\lk(X,v)$, let $u$ be the unique vertex in $\sigma\setminus\tau$. Since $v\tau\in X$ but $v\sigma\notin X$, $u$ must belong to every missing face of $X$ contained in $v\sigma$. Also $v$ must belong to every such missing face, since $\sigma\in X$. Therefore, since all the missing faces of $X$ are of size at most $d+1$, there can be at most $d$ such different vertices $u$, so
		\[
		\left|\set{\tau\in \faces{k-1}{\sigma} : \, \tau\in\lk(X,v)}\right|\leq d.
		\]
		Thus we obtain	
		\begin{align*}
		\sum_{\tau\in\faces{k-1}{\sigma}} \degree{X}(\tau)&\leq (k+1)(\degree{X}(\sigma)+1)+\sum_{\substack{v\in V\setminus\sigma,\\ v\notin\lk(X,\sigma)}}d\\
		&\leq (k+1)(\degree{X}(\sigma)+1)+(n-k-1-\degree{X}(\sigma))d\\
		&= d n -(d-1)(k+1) +(k-d+1)\degree{X}(\sigma).
		\end{align*}
\end{proof}

\section{Main results}
\label{section:main}
In this section we prove our main results, Theorem \ref{thm:gersgorin_for_laplacian} and its corollary Theorem \ref{cor:gersgorin_for_laplacian}.

\begin{proof}[Proof of Theorem \ref{thm:gersgorin_for_laplacian}]
	
	For $k=-1$ the claim holds since $\delta_{-1}(X)=\mineig{-1}{X}=n$.
	Assume now $k\geq 0$. By Claim \ref{claim:lapmatrix}, we have for $\sigma\in\faces{k}{X}$ 
	\[
	\lap{k}(\sigma,\sigma)=\degree{X}(\sigma)+k+1
	\]
	and
	\begin{align}
	\label{eq:sumoffdiagonal}
	\sum_{\substack{\eta\in\faces{k}{X},\\ \eta\neq \sigma}}& \left|\lap{k}(\sigma,\eta)\right|=
	\left|\set{\eta\in\faces{k}{X} : \left|\sigma\cap\eta\right|=k,\, \sigma\cup\eta\notin\faces{k+1}{X}}\right|\nonumber\\
	&=\sum_{\tau\in\faces{k-1}{\sigma}}\left|\set{v\in V\setminus\sigma :\,
		v\in\lk(X,\tau),\, v\notin\lk(X,\sigma)}\right|\nonumber\\
	&=\sum_{\tau\in\faces{k-1}{\sigma}}\left(\degree{X}(\tau)-1-\degree{X}(\sigma)\right)
	\nonumber\\
	&=\sum_{\tau\in\faces{k-1}{\sigma}}\degree{X}(\tau) -(k+1)(\degree{X}(\sigma)+1).
	\end{align}
	So by Ger\v sgorin's theorem (Theorem \ref{thm:gers}) we obtain
	\begin{multline}
	\label{eq:gersbound1}
	\mineig{k}{X}\geq \min_{\sigma\in\faces{k}{X}}\left( \lap{k}(\sigma,\sigma)-\sum_{\substack{\eta\in\faces{k}{X},\\ \eta\neq \sigma}} \left|\lap{k}(\sigma,\eta)\right|  \right)\\
	=\min_{\sigma\in\faces{k}{X}}\left( \degree{X}(\sigma)+k+1-\sum_{\tau\in\faces{k-1}{\sigma}}\degree{X}(\tau)+(k+1)(\degree{X}(\sigma)+1)\right)\\
	=\min_{\sigma\in\faces{k}{X}}\left( (k+2)\degree{X}(\sigma)+2(k+1)-\sum_{\tau\in\faces{k-1}{\sigma}}\degree{X}(\tau) \right).
	\end{multline}
	Recall that by Lemma \ref{lem:countdegrees_ver2} we have 
	\begin{equation}
	\label{eq:tempsumdegree}
			\sum_{\tau\in\faces{k-1}{\sigma}} \degree{X}(\tau)-(k-d+1)\degree{X}(\sigma)
			\leq dn-(d-1)(k+1).
	\end{equation}
	Combining \eqref{eq:gersbound1} and \eqref{eq:tempsumdegree} we obtain
	\begin{multline*}
	\mineig{k}{X}\geq 
	\min_{\sigma\in\faces{k}{X}}\left( (d+1)(\degree{X}(\sigma)+k+1)-dn \right)
	\\=(d+1)(\delta_k+k+1)-dn,
	\end{multline*}
as wanted.
\end{proof}

\begin{remark}
	A slightly different approach to the proof of Theorem \ref{thm:gersgorin_for_laplacian} is the following: We build a graph $G_k$ on vertex set $V_k=X(k)$, with edge set
	\[
	E_k=\{ \{\sigma,\tau\} : \,\, \sigma,\tau\in X(k), \, |\sigma\cap\tau|=k,\, \sigma\cup\tau\notin X(k+1)\}.
	\]
	We make $G_k$ into a signed graph (see \cite{zaslavsky2013matrices}) by defining the sign function $\phi: E_k\to \{-1,+1\}$ by
	\[
	\phi(\{\sigma,\tau\})= -\sgn{\sigma}{\sigma\cap\tau}\sgn{\tau}{\sigma\cap\tau}.
	\]
	The incidence matrix $H_k$ is the $V_k\times E_k$ matrix
	\[
	H_k(\sigma,\{\eta,\tau\})=\begin{cases}
	\sgn{\sigma}{\eta\cap\tau} & \text{ if } \sigma\in \{\eta,\tau\},\\
	0 & \text{ otherwise}.
	\end{cases}
	\]
	Define the Laplacian of $G_k$ to be the $V_k\times V_k$ matrix $K_k=H_k H_k^T$. So $K_k$ is positive semi-definite, and we have
	\begin{multline*}
	K_k (\sigma,\tau)=
	\begin{cases} 
	\deg_{G_k}(\sigma) & \text{if } \sigma=\tau,\\
	-\phi(\{\sigma,\tau\}) & \text{if } \{\sigma,\tau\}\in E_k,\\
	0 &\text{otherwise.}									
	\end{cases}\\
	=	\begin{cases} 
	|\{ \eta\in X(k) : \, |\sigma\cap\eta|=k, \, \sigma\cup\eta\notin X(k+1)\}| & \text{if } \sigma=\tau,\\
	\sgn{\sigma}{\sigma\cap\tau}\sgn{\tau}{\sigma\cap\tau} & \text{if } \substack{ |\sigma\cap\tau|=k,\,\, \sigma\cup\tau\notin X(k+1),}\\
	0 &\text{otherwise.}									
	\end{cases}
	\end{multline*}
	By Equation \eqref{eq:sumoffdiagonal} we obtain
	\[
		L_k= D_k + K_k,
	\]
	where $D_k$ is the diagonal matrix with diagonal elements
	\[
		D_k(\sigma,\sigma)=2(k+1)+(k+2) \degree{X}(\sigma)-\sum_{\tau\in\sigma(k-1)} \degree{X}(\tau).
	\]
	This decomposition of the Laplacian first appeared in \cite{forman2003bochner}, where the graph Laplacian $K_k$ is called the \emph{Bochner Laplacian} of $X$.
	
	Using the fact that $K_k$ is positive semi-definite, and applying Lemma \ref{lem:countdegrees_ver2} as before, we obtain Theorem \ref{thm:gersgorin_for_laplacian}.	
\end{remark}

Now we can prove Theorem \ref{cor:gersgorin_for_laplacian}:
\begin{proof}[Proof of Theorem \ref{cor:gersgorin_for_laplacian}]
	Let $k>\frac{d}{d+1}{n}-1$. By Theorem \ref{thm:gersgorin_for_laplacian} we have
	\[
	\mineig{k}{X}\geq (d+1)(k+1)-dn > (d+1)\frac{d}{d+1}n-dn=0.
	\]
	So by the simplicial Hodge theorem (Corollary \ref{cor:hodge}), $\cohomology{k}{X}=0$.	
\end{proof}

\section{Extremal examples}
\label{section:extremal}

	In this section we prove Propositions \ref{prop:extremalexample} and \ref{thm:gers_for_laplacian_eq_case} about complexes achieving equality in Theorem \ref{thm:gersgorin_for_laplacian}.

\begin{proof}[Proof of Proposition \ref{prop:extremalexample}]
	Let $-1\leq k\leq dt+r-1$.
	By Theorem \ref{thm:joinlaplacian} we have
	\[
	\mineig{k}{Z}= \min \left\{ \mineig{i_1}{\skeleton{d}{d-1}}+\cdots+\mineig{i_t}{\skeleton{d}{d-1}}+\mineig{j}{\simplex{r-1}} : \substack{ -1\leq i_1,\ldots,i_t\leq d-1,\\ -1\leq j\leq r-1,\\ i_1+\cdots+i_t+j=k-t}\right\}.
	\]
	By Claim \ref{claim:complete_skeleton_spectrum} we have
	\[
	\mineig{j}{\skeleton{d}{d-1}}=\begin{cases}
	d+1 & \text{ if } -1\leq j\leq d-2,\\
	0   & \text{ if }  \quad\quad\quad j=d-1,
	\end{cases}
	\]
	and $\mineig{j}{\simplex{r-1}}=r$ for all $-1\leq j\leq r-1$.
	Therefore $\mineig{k}{Z}=(d+1)(t-m)+r$, where $m$ is the maximal number of indices in $i_1,\ldots,i_t$ that can be chosen to be equal to $d-1$. That is, $m$ is the maximal integer between $0$ and $t$ such that there exist  $-1\leq i_1,\ldots,i_{t-m}\leq d-2$ and $-1\leq j \leq r-1$ satisfying
	\[
	m(d-1)+i_1+\cdots+i_{t-m}+j=k-t.
	\]
	We obtain
	\[
	m=\begin{cases}
	\left\lfloor \frac{k+1}{d} \right\rfloor & \text{ if } -1\leq k\leq dt-1, \\
	\,\,\,\,t & \text{ if }\,\,\,\,\, dt\leq k\leq dt+r-1.
	\end{cases}
	\]
	So
	\[
	\mineig{k}{Z}=\begin{cases}
	(d+1)\left(t-\left\lfloor \frac{k+1}{d} \right\rfloor\right)+r & \text{ if } -1\leq k\leq dt-1, \\
	r & \text{ if }\,\,\,\,\, dt\leq k\leq dt+r-1.
	\end{cases}
	\]
	
	Next we consider the degrees of simplices in $Z$: 
	Let $V$ be the vertex set of $Z$. Recall that $|V|=n=(d+1)t+r$.	
	Let $V_1,\ldots, V_t\subset V$ be the vertex sets of the $t$ copies of $\skeleton{d}{d-1}$.
	
	Let $\sigma\in \faces{k}{Z}$. A vertex $v\in V\setminus\sigma$ belongs to $\lk(Z,\sigma)$ unless $V_i\subset \sigma\cup\{v\}$ for some $i\in [t]$ (see Figure \ref{fig:degrees}). Therefore	
	 $\degree{Z}(\sigma)=n-(k+1)-s(\sigma)$, where
	\[
	s(\sigma)=\left|\set{ i\in[t] : \, \left|\sigma\cap V_i\right|=d}\right|.
	\]
	So the minimal degree of a simplex in $\faces{k}{Z}$ is
	\[
	\delta_k(Z)=\begin{cases}
	n-(k+1)-\left\lfloor \frac{k+1}{d} \right\rfloor & \text{ if } -1\leq k\leq dt-1, \\
	\,\,\,\,n-(k+1)-t & \text{ if }\,\,\,\,\, dt\leq k\leq dt+r-1.
	\end{cases}
	\]
	Therefore $\delta_k(Z)=n-(k+1)-m$, thus
	\[
	(d+1)(\delta_k(Z)+k+1)-dn= n- (d+1)m= \mineig{k}{Z}.
		\]
		
	\begin{figure}
		\begin{center}
				\includegraphics[scale=0.8]{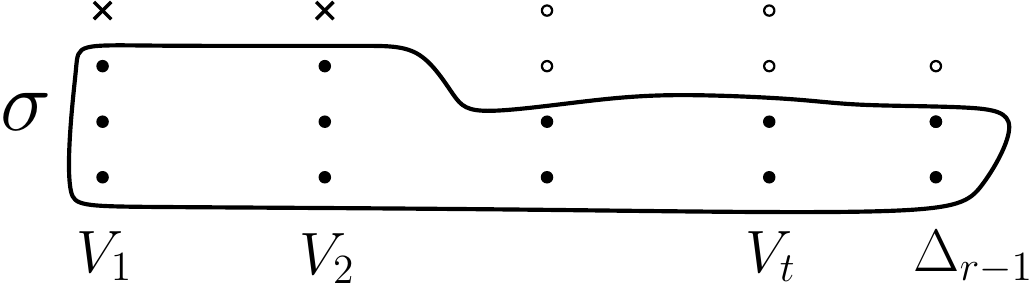}
		\end{center}
		\caption{A simplex $\sigma\in Z$. The black dots are the vertices in $\sigma$. The white dots are the vertices in $\lk(Z,\sigma)$. The crosses are the vertices in $V\setminus\sigma$ that do not belong to $\lk(Z,\sigma)$ (these are the vertices that, when added to $\sigma$, complete a missing face).}
		\label{fig:degrees}
	\end{figure}

\end{proof}

\begin{proof}[Proof of Proposition \ref{thm:gers_for_laplacian_eq_case}]

	Let $\sigma_0\in X(k)$ such that
	\begin{multline*}
		\min_{\sigma\in\faces{k}{X}}\left( (k+2)\degree{X}(\sigma)+2(k+1)-\sum_{\tau\in\faces{k-1}{\sigma}}\degree{X}(\tau) \right)\\=
		(k+2)\degree{X}(\sigma_0)+2(k+1)-\sum_{\tau\in\faces{k-1}{\sigma_0}}\degree{X}(\tau).
	\end{multline*}
	By Inequality \eqref{eq:gersbound1} and Lemma \eqref{lem:countdegrees_ver2} we have	
	\begin{align*}
	2(k+1)-n&=\mineig{k}{X}\\
	&\geq 
	\min_{\sigma\in\faces{k}{X}}\left( (k+2)\degree{X}(\sigma)+2(k+1)-\sum_{\tau\in\faces{k-1}{\sigma}}\degree{X}(\tau) \right)\\
	&= 
	 (k+2)\degree{X}(\sigma_0)+2(k+1)-\sum_{\tau\in\faces{k-1}{\sigma_0}}\degree{X}(\tau).
	\\
	&\geq
	 2(\degree{X}(\sigma_0)+k+1)-n \\
	&\geq 2(k+1)-n.
	\end{align*}
	So all the inequalities are actually equalities, therefore we obtain
	\[
		\degree{X}(\sigma_0)=0
	\]
	and
	\[
	\sum_{\tau\in\faces{k-1}{\sigma_0}}\degree{X}(\tau)=n.
	\]
	By Claim \ref{lem:countdegrees_part1} we have
		\begin{multline}\label{eq:temp1}
		n=\sum_{\tau\in\faces{k-1}{\sigma_0}}\degree{X}(\tau)= k+1
	+	\sum_{\substack{v\in V \setminus\sigma_0,\\ v\notin\lk(X,\sigma_0)}} \left|\{\tau\in \sigma_0(k-1) : \, v\in\lk(X,\tau)\}\right|.
		\end{multline}
	Let $v\in V\setminus \sigma_0$ and $\tau\in \sigma_0(k-1)$. Denote $\{u\}=\sigma_0\setminus\tau$. We have 
	\begin{align*}
	\begin{matrix}
	& 	 v\notin\lk(X,\sigma_0) \text{ and  }\\ 
	&  v\in \lk(X,\tau)
	\end{matrix}
	\iff
	\begin{matrix}
	& \text{ the only missing face of $X$ contained} \\ 
	& \text{ in $v\sigma_0$ is the edge $uv$.}
	\end{matrix}
	\end{align*}
	Thus, if $v\notin \lk(X,\sigma_0)$ and $v\in\lk(X,\tau)$, then for any $\tau\neq \tau'\in \faces{k-1}{\sigma_0}$ we must have $v\notin \lk(X,\tau')$. 	
	 Denote by $\tempset(\sigma_0)$ the set of vertices $v\in V\setminus \sigma_0$ such that $v\sigma_0$ contains only one missing face. We obtain:
	 \[
		 	\sum_{\substack{v\in V \setminus\sigma_0,\\ v\notin\lk(X,\sigma_0)}} \left|\{\tau\in \sigma_0(k-1) : \, v\in\lk(X,\tau)\}\right|= |Q(\sigma_0)|,
	\]
	therefore by Equation \eqref{eq:temp1} we have
		\[
			n
			= (k+1)+|\tempset(\sigma_0)|= |\sigma_0|+|\tempset(\sigma_0)|.
		\]	
	So $|\tempset(\sigma_0)|= n-|\sigma_0|=|V\setminus\sigma_0|$,
	thus $\tempset(\sigma_0)=V\setminus\sigma_0$.
	Hence, for every vertex $v\in V\setminus\sigma_0$ there is exactly one vertex $u\in \sigma_0$ such that $uv\notin \faces{1}{X}$.
	
	Denote the vertices in $V\setminus\sigma_0$ by $v_1,\ldots,v_{n-k-1}$. For each $v_i$ denote by $u_i$ the unique vertex in $\sigma_0$ such that $u_i v_i\notin\faces{1}{X}$.
	
	Let $A=\sigma_0\setminus\set{u_1,\ldots,u_{n-k-1}}$ and $r=|A|$. Each vertex in $A$ is connected in the graph $X(1)$ to any other vertex. Therefore, since $X$ is a clique complex, we have $X=X[A]*Y$, where $Y=X[V\setminus A]$.

	But $X[A]\cong \simplex{r-1}$, therefore $\mineig{i}{X[A]}=r$ for all $-1\leq i\leq r-1$, so by Theorem \ref{thm:joinlaplacian}:
	\begin{multline}
	\label{eq:eq_gers3}
	\mineig{k}{X}=\min_{i+j=k-1} \big(\mineig{i}{X[A]}+\mineig{j}{Y}\big)\\
	= r + \min\bigg\{\mineig{j}{Y}: \,\max\{-1,k-r\}\leq j\leq \min\{k,\dim(Y)\}\bigg\} 
	\geq r.
	\end{multline}
	If the vertices $u_1,\ldots,u_{n-k-1}$ are not all distinct, then we have
	\[
	r>|\sigma_0|-(n-k-1)=2(k+1)-n=\mineig{k}{X},
	\]
	a contradiction to \eqref{eq:eq_gers3}. Therefore $u_1,\ldots,u_{n-k-1}$ are all different vertices (see Figure \ref{fig:vertices}), so $r=\mineig{k}{X}=2(k+1)-n$. This implies that the inequality in \eqref{eq:eq_gers3} is an equality. Hence, there exists some $j\geq k-r=n-k-2$ such that $\mineig{j}{Y}=0$.
	Let \[Y'=\{u_1, v_1\}*\{u_2,v_2\}*\cdots*\{u_{n-k-1},v_{n-k-1}\}\cong\left(\skeleton{1}{0}\right)^{*(n-k-1)}.\]
		The geometric realization of $Y'$ is the boundary of the $(n-k-1)$-dimensional cross-polytope, so $Y'$ is a triangulation of the $(n-k-2)$-dimensional sphere. We have $Y\subset Y'$, therefore $\dim(Y)\leq \dim(Y')=n-k-2$. Thus we must have $\mineig{n-k-2}{Y}=0$, so by the simplicial Hodge theorem (Corollary \ref{cor:hodge}), $\cohomology{n-k-2}{Y}\neq 0$. But any proper subcomplex of $Y'$ has trivial $(n-k-2)$-dimensional cohomology, therefore $Y=Y'$.
	Hence, \[X\cong\left(\skeleton{1}{0}\right)^{*(n-k-1)}*\simplex{2(k+1)-n-1}.\]

	\begin{figure}
		\begin{center}
			\includegraphics[scale=0.8]{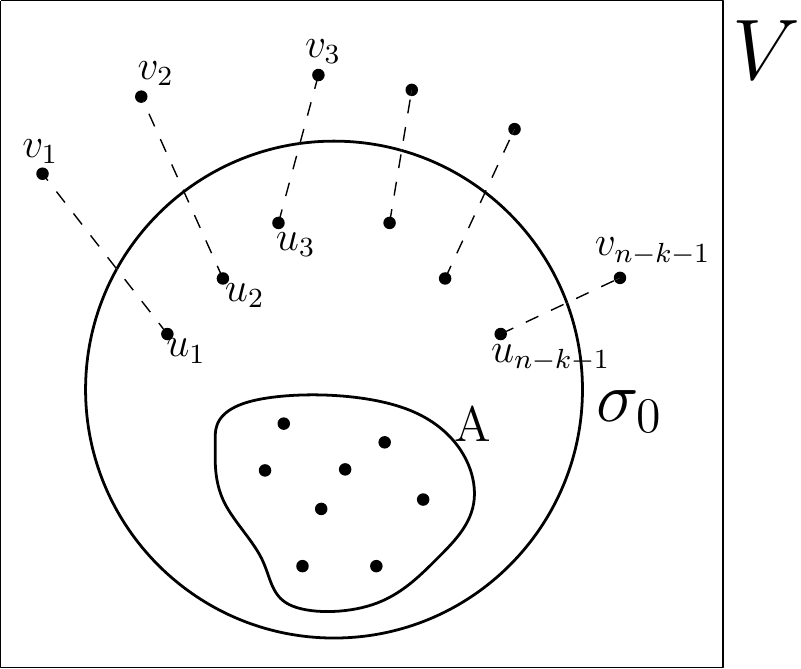}
		\end{center}
		\caption{The vertices in $V$. Each vertex $v_i\in V\setminus \sigma_0$ is connected to all the vertices in $\sigma_0$ except $u_i$. Each vertex in $A$ is connected to every other vertex in $V$.}
		\label{fig:vertices}
	\end{figure}
	
\end{proof}

Proposition \ref{thm:gers_for_laplacian_eq_case} characterizes, for the case of clique complexes ($h(X)=d=1$), the complexes achieving the equality
\[
	\mu_k(X)= (d+1)(k+1)-dn
\]
at some dimension $k$. It would be interesting to extend this characterization to complexes with higher dimensional missing faces.
We expect the situation to be similar to the case $d=1$, that is:
	\begin{conjecture}
		\label{conj:gers_for_laplacian_eq_case}
		Let $X$ be a simplicial complex on vertex set $V$ of size $n$, with $h(X)=d$, such that   $\mineig{k}{X}=(d+1)(k+1)-dn$ for some $k$. Then
		\[
		X\cong\left(\skeleton{d}{d-1}\right)^{*(n-k-1)}*\simplex{(d+1)(k+1)-dn-1},
		\]
		(and in particular, $\dim(X)=k$).
	\end{conjecture}

\section*{Acknowledgment}
This paper was written as part of my M. Sc. thesis, under the supervision of Professor Roy Meshulam. I thank Professor Meshulam for his guidance and help.

\end{document}